\documentclass[crop,referee,a4paper,sn-mathphys-num]{sn-jnl}

\usepackage{lmodern}
\usepackage[T1]{fontenc}
\usepackage[utf8]{inputenc}
\usepackage{amsmath,amssymb,amsfonts}%
\usepackage{amsthm}%
\usepackage{algorithm, algpseudocode}
\usepackage{graphicx}
\usepackage{mathtools,xfrac}
\usepackage{mathrsfs}
\usepackage{dsfont}

\usepackage{xcolor}

\usepackage{booktabs} 
\usepackage{siunitx}

\usepackage{hyperref}
\hypersetup{
  colorlinks=true}

\usepackage[capitalize,nameinlink]{cleveref}
\crefformat{equation}{\textup{#2(#1)#3}}
\crefrangeformat{equation}{\textup{#3(#1)#4--#5(#2)#6}}
\crefmultiformat{equation}{\textup{#2(#1)#3}}{ and \textup{#2(#1)#3}}
{, \textup{#2(#1)#3}}{, and \textup{#2(#1)#3}}
\crefrangemultiformat{equation}{\textup{#3(#1)#4--#5(#2)#6}}%
{ and \textup{#3(#1)#4--#5(#2)#6}}{, \textup{#3(#1)#4--#5(#2)#6}}{, and \textup{#3(#1)#4--#5(#2)#6}}

\Crefformat{equation}{#2Equation~\textup{(#1)}#3}
\Crefrangeformat{equation}{Equations~\textup{#3(#1)#4--#5(#2)#6}}
\Crefmultiformat{equation}{Equations~\textup{#2(#1)#3}}{ and \textup{#2(#1)#3}}
{, \textup{#2(#1)#3}}{, and \textup{#2(#1)#3}}
\Crefrangemultiformat{equation}{Equations~\textup{#3(#1)#4--#5(#2)#6}}%
{ and \textup{#3(#1)#4--#5(#2)#6}}{, \textup{#3(#1)#4--#5(#2)#6}}{, and \textup{#3(#1)#4--#5(#2)#6}}

\theoremstyle{thmstyleone}%
\newtheorem{theorem}{Theorem}[section]
\newtheorem{proposition}[theorem]{Proposition}%
\theoremstyle{thmstyletwo}%
\newtheorem{example}[theorem]{Example}%

\newtheorem{remark}[theorem]{Remark}%

\theoremstyle{thmstylethree}%
\newtheorem{definition}[theorem]{Definition}%

\crefname{example}{Example}{Examples}
\usepackage{autonum}

\DeclareMathOperator*{\argmin}{arg\,min} 

\DeclareMathOperator{\cl}{cl} 
\DeclareMathOperator{\conv}{conv}

\newcommand{\norm}[1]{\left \lVert #1 \right \rVert}

\newcommand{\closure}[1]{\cl\left({ #1 }\right)}

\def\RR{\mathds R}
\def\NN{\mathds N}

\def\HH{\mathcal H}

\def\lV{\left\lVert }
\def\rV{\right\rVert }

\DeclareMathOperator{\aff}{aff}

\DeclareMathOperator{\inte}{int}

\newcommand{\scal}[2]{\left\langle #1, #2 \right\rangle}

\usepackage[shortlabels]{enumitem}
\newlist{lista}{enumerate}{1}
\setlist[lista]{label=\alph*., nosep,leftmargin=*,align=right}

\newlist{listi}{enumerate}{1}
\setlist[listi]{label={\upshape(\roman*\upshape)},leftmargin=*,align=right, widest=iii,nosep, format=\bf}

\definecolor{ad}{RGB}{0,0,250}

\begin{document}

\title{Fejér* monotonicity in optimization algorithms}

\author[1]{Roger Behling}\email{rogerbehling@gmail.com}
        \author[2]{Yunier Bello-Cruz}\email{yunierbello@niu.edu}
               \author[3]{Alfredo Noel Iusem}\email{alfredo.iusem@fgv.br}
       \author[4]{Ademir Alves Ribeiro}\email{ademir.ribeiro@ufpr.br}
        \author*[1]{Luiz-Rafael Santos}\email{l.r.santos@ufsc.br}

        \affil[1]{{Departamento de Matemática}, {Universidade Federal de Santa Catarina}, {{Blumenau}-{SC} -- {89065-300}, {Brazil}}}
          \affil[2]{{{Department of Mathematical Sciences}, {Northern Illinois University}, {{DeKalb}-{IL} -- {60115-2828}, {USA}} }}
          \affil[3]{{Escola de Matemática Aplicada, Fundação Getúlio Vargas, {{Rio de Janeiro}-{RJ} -- {22250-900}, {Brazil}}}}
           \affil[4]{Departamento de Matemática, Universidade Federal do Paraná, 
           Curitiba-PR -- {81531-980}, Brazil}



\abstract{
Fejér monotonicity is a well-established property often observed in sequences generated by optimization algorithms. In this paper, we study an extension of this property, called Fejér* monotonicity, which was initially proposed in [\emph{SIAM J. Optim.}, 34(3), 2535--2556 (2024)]. We discuss and explore its behavior within Hilbert  spaces as a tool for optimization algorithms. Additionally, we investigate weak and strong convergence properties of this novel concept. Through   illustrative examples and insightful results, we contrast Fejér* with weaker notions of quasi-Fejér-type monotonicity. 
}

        \maketitle

\section{Introduction}\label{sec:intro}

A useful concept for analyzing the convergence of optimization algorithms is the notion of \emph{Fejér monotonicity}. We denote by $\NN$ the set of all nonnegative integers.
A sequence $(x^k)_{k\in \NN}$ defined in a real Hilbert space $\HH$ is said to be \emph{Fejér monotone} with respect to a {nonempty} set $M\subset\HH$ if for all $x\in M$ and all $k\in\NN$ we have
\(\label{eq:def.Fejermonotone}
\lVert x^{k+1}-x\rVert\leq\lVert x^k-x\rVert\). It is well-known that if $(x^k)_{k\in \NN}$ is Fejér monotone with respect to $M$, then $(x^k)_{k\in \NN}$ is bounded, and if all its weak cluster points belong to  {$M$}, then the entire sequence $(x^k)_{k\in \NN}$ converges weakly to some point lying in $M$; see, for instance, \cite[Thm~2.16]{Bauschke:1996}, \cite[Prop.~5.4 and Thm.~5.5, p.~92]{Bauschke:2017a}, and  \cite{Combettes:2009} for further details and historical notes.

The notion of Fejér monotonicity is named after Hungarian mathematician Lipót Fejér (1880-1959) who introduced this notion in \cite{Fejer:1922} and further developed it in~\cite{Fejer:1937} and later in \cite{Agmon:1954,Motzkin:1954}.  Fejér was the force behind the highly successful Hungarian school of analysis. His influence on the development of mathematics in the last century  should be acknowledged. Indeed, Fejér was the advisor of mathematicians such as John von Neumann, Paul Erdös, George Pólya, Pál Turán and Gábor Szegö.

Weaker notions of Fejér monotonicity have been considered in the literature. For instance,  the concept of \emph{quasi-Fejér monotonicity}, first presented in \cite{Ermolev:1969} and further elaborated in several works such as  \cite{Iusem:1994,Combettes:2001a}.

In this work, we  explore a slightly  different notion, called \emph{Fejér* monotonicity}. It was first introduced in~\cite{Behling:2024c} as a key ingredient to show that the Circumcentered-Reflection Method (CRM)~\cite{Behling:2024b}, proposed by the authors of~\cite{Behling:2024c}, solves the Convex Feasibility Problem (CFP). The main difference with the usual Fejér monotonicity notion relies on the fact that now the  {nonincreasing} distance property holds for the tail of the sequence, starting at some index which depends on the considered point $x\in M$.

In this work we further explore this novel concept. 
Indeed, the main contributions of this paper are twofold: (i) We present a detailed discussion on the properties of Fejér* monotonicity, including the ones that can be used as a tool to analyze convergence of sequences generated by optimization algorithms as well as the relationship of Fejér*  with other weaker notions of Fejér monotonicity; (ii) We provide an illustrative example that highlights the differences between Fejér* and quasi-Fejér monotonicity. This example is key  to clarify the connections and differences of these concepts.

This paper is organized as follows. We start in \Cref{sec:fejer*optimization} with a brief discussion on the role of Fejér monotonicity in optimization algorithms.
In \Cref{sec:fejer*} we present some basic properties of Fejér* monotonicity, and we discuss, upon illustrative examples, some properties of Fejér* monotone sequences. In \Cref{sec:fejer-quasi-fejer} we compare 
Fejér* sequences and quasi-Fejér sequences. Finally, in \Cref{sec:conlcuding} we present some concluding remarks.

\section{Fejér* monotonicity in optimization algorithms}\label{sec:fejer*optimization}

First, let us formally introduce the notion of Fejér* monotonicity.

\begin{definition}[Fejér* monotonicity]
\label{def.Fejer*monotone}
    Let $M\subset \HH$ {be a nonempty set} and consider a sequence $(x^k)_{k\in \NN} \subset \HH$.  We say that $(x^k)_{k\in \NN}$ is \emph{Fejér* monotone with respect to $M$} if for any point $x \in M$, there exists $N(x)\in\NN$ such that, for all $k\ge N(x)$,
         \[\label{eq:def.Fejer*monotone}
        \| x^{k+1}-x\|\le\| x^k -x\|.
    \] 
\end{definition}

Note that Fejér monotonicity implies Fejér* monotonicity, but the converse is not true in general. We address now the role of Fejér and Fejér*  monotonicity in connection with fixed-point and feasibility problems. Consider an operator $T:\HH\to\HH$. We remind that $T$ is a \emph{contraction} if there exists $\tau\in (0,1)$ such that 
\[\label{e1}
\lV T(x)-T(y)\rV\le\tau\lV x-y\rV,
\]
for all $x,y\in \HH$. 
It is well known that a contraction $T$ has a unique fixed point (\emph{i.e.}, a point $\bar x\in\HH$ such that $T(\bar x)=\bar x$) and that the sequence obtained by iterating the operator $T$ (\emph{i.e.}, the sequence 
$(x^k)_{k\in\NN}\subset\HH$ defined as $x^{k+1}=T(x^k)$ starting at any $x^0\in\HH$), converges to such a unique fixed point; for further details on contraction mappings, see~\cite{Kirk:2001}.

The situation changes when we consider notions which are weaker than
the contraction property. The most basic one is that of {\it nonexpansiveness}. An operator 
$T$ is \emph{nonexpansive} if
\[\label{e2}
\lV T(x)-T(y)\rV\le\lV x-y\rV,
\]
for all $x,y\in\HH$.

The set of fixed points of a nonexpansive operator $T$ may be empty, or infinite,
and the behavior of the sequence $(x^k)_{k\in\NN}$ obtained by iterating $T$ may fail to converge,
but given any fixed point $\bar x$ of $T$, it follows immediately from \cref{e2} that $\lV x^{k+1}-\bar x\rV\le\lV x^k-\bar x\rV$, \emph{i.e.}, the sequence is Fejér monotone with respect to the set $M$ of fixed points of $T$, and thus 
it enjoys the basic properties of Fejér monotone sequences, namely that it is bounded, and that if all its cluster points belong to $M$, then the sequence $(x^k)_{k\in\NN}$ converges to one of them; see, for instance, \cite{Bauschke:2017a} for  more on nonexpansiveness and Fejér monotonicity.

Among the nonexpansive operators, {\it projections} onto convex sets are quite significant. We remind that, given a nonempty, convex and closed set $C\subset\HH$, the projection onto $C$, $P_C:\HH\to C$, is defined as 
$P_C(x)\coloneqq \argmin_{y\in C}\lV x-y\rV$. Convex combinations of projections (\emph{i.e.}, operators of the form $\bar P=\sum_{i=1}^m \alpha_iP_{C_i}$ with $\alpha_i\in [0,1], \sum_{i=1}^m\alpha_i=1$, and $C_i\subset\HH$ convex) are nonexpansive, and the same holds for compositions of projections, \emph{i.e.},
operator of the form $\widehat P=P_{C_m}\circ \dots\circ P_{C_1}$. 
The set of fixed points of a projection $P_C$ is $C$, and the sets of fixed points of $\bar P$ and $\widehat P$, as defined above, are equal to the set  $C:=\cap_{i=1}^mC_i$ whenever $C\neq \emptyset$. Iterating the operators $\bar P$ and $\widehat P$, one obtains the so-called {\it Simultaneous Projections} and {\it Sequential Projections} algorithms, which aim to solve the CFP, consisting of finding a point in the intersection $C$ of nonempty, convex and closed sets $C_1, \dots, C_m$. See the work of \citet{Bauschke:1996} for a detailed discussion on these methods, which can be traced back to the seminal works of \citet{Cimmino:1938} and \citet{Kaczmarz:1937}. In view of the discussion above, the sequences generated by these two methods are Fejér monotone with respect to $C$. 

In order to show that the sequences generated by these algorithms do converge to a point in $C$, one needs a bit more than nonexpansiveness. Two suitable properties, slightly stronger than nonexpansiveness but weaker than the contraction property, are
\begin{listi}
\item \emph{firm nonexpansiveness}, meaning that
\[\label{e3}
\lV T(x)-T(y)\rV^2\le\lV x-y\rV^2-\lV (T(x)-T(y))-(x-y)\rV^2
\]
for all $x,y\in\HH$; and
\item
 \emph{nonexpansiveness plus}, meaning that
\[\label{e4}
\lV T(x)-T(y)\rV \le\lV x-y\rV
\]
for all $x,y\in\HH$, and $T(x)-T(y)=x-y$ whenever equality holds in
\cref{e4}.
\end{listi}
Projections onto nonempty, convex and closed sets, as well as the above defined operator $\bar P$, are firmly nonexpansive. With the help of Fejér convergence properties and the notion of firm nonexpansiveness, it is easy to prove that the Simultaneous and the Successive Projection algorithms do converge to a point in $C$, whenever $C$ is nonempty (see 
\cite{Bauschke:1996}).

One way to accelerate these methods for CFP is to  approximate the set $C$ from the inside, \emph{i.e.}, to replace at iteration $k$ the set $C$ by a set $C^k$ such that
\[\label{e5}
C^k\subset C^{k+1}\subset  \dots\subset C,
\]
for all $k$. Then one may consider
the sequence $(x^k)_{k\in\NN}$ given by $x^{k+1}\coloneqq T_k(x^k)$, where $T_k$ is the projection onto a nonempty, convex and closed set which contains $C^k$ (but may not be contained in, or even contain $C$). 
In such a case, the nonexpansiveness of $T_k$ gives $\lV x^{k+1}-y\rV\le \lV x^k-y\rV$, but only for points $y\in C^k$, not in the whole set $C$. 

For instance, suppose that we have $C\coloneqq\{x\in\HH\mid  g(x) \le 0\}$,  for some function $g:\HH\to\RR$ that is strictly convex. This choice is consistent with the \emph{Slater hypothesis} (\emph{i.e.}, the existence of $\hat y \in \HH$, such that $g(\hat y)<0$) that is needed in this context. Let us define 
\[\label{e6}
C^k\coloneqq\{x\in\HH\mid g(x)+\epsilon_k\le 0\},
\] 
where $(\epsilon_k)_{k\in \NN}$ is a sequence of positive  real numbers decreasing to $0$.   If one assumes Slater, the set $C$ has nonempty interior, \emph{i.e.}, $\inte(C) \neq \emptyset$, and then any point $\hat y\in \inte (C)$ lies in $C^k$, for all $k$ greater than $k(\hat y)$, that is, $C^k\neq \emptyset$. This  $k(\hat y)$ is defined as the first integer such that $\epsilon_{k(\hat y)}<-g(\hat y)$, due to \cref{e5}. This is exactly the same as saying that $(x^k)_{k\in\NN}$ is Fejér* monotone with respect to $\inte(C)$. Hence, Fejér* monotonicity is precisely what is needed for the analysis of algorithms which use approximation schemes as in \cref{e5}.

Among them, we mention the algorithms in \cite{fukushimaFinitelyConvergentAlgorithm1982,Iusem:1986a,iusemFinitelyConvergentIterative1987,DePierro:1988a}, which were proposed
for solving CFP. These methods employ the scheme in \cref{e5} combined with operators related to $\bar P$ and $\widehat P$, where the projections onto the sets $C_i$, 
approximated from the inside by sets $C^k_i$, as in \cref{e6}, are replaced by projections onto hyperplanes separating $x^k$ from $C^k_i$.  We stress here that the main  Theorem in \cite[p.~1126]{fukushimaFinitelyConvergentAlgorithm1982}, Theorem 1 in \cite{iusemFinitelyConvergentIterative1987}, and Lemma 1 in  \cite{DePierro:1988a}   clearly state that only the tail of the sequence generated by  these author's algorithms are Fejér monotone, starting at some index which depends on the considered point. Therefore, the sequences are indeed Fejér* monotone.

A similar situation occurs with the  algorithm in \cite{Behling:2024c}, where the algorithm in \cite{Iusem:1986a} is enhanced through the addition of a circumcentering step, using the machinery developed during the last few years in \cite{Araujo:2022, Arefidamghani:2021, Arefidamghani:2023, Behling:2018, Behling:2018a, Behling:2020, Behling:2021a, Behling:2023, Behling:2024, Behling:2024b, Behling:2024c}. It is worthwhile to note that in the case of \cite{Iusem:1986a,Behling:2024c,DePierro:1988a,iusemFinitelyConvergentIterative1987,fukushimaFinitelyConvergentAlgorithm1982}, under the Slater hypothesis, if the $\epsilon_k$'s decrease to zero slowly enough, \emph{e.g.}, so that
$\sum_{k=0}^{\infty}\epsilon_k=+\infty$ (take, for instance, $\epsilon_k=1/k$), convergence is indeed finite, \emph{i.e.}, there exists $\hat k$ such that $x^{\hat k}$ belongs to 
$C=\bigcap_{i=1}^m C_i$.

More recently, \citet{Kolobov:2021,Kolobov:2022} also make use of the  perturbed approximation scheme \cref{e5} in their algorithms to solve CFPs. In particular, \citet[in Theorem 3.1]{Kolobov:2021} also shows that the tail of the sequence generated by their method is Fejér monotone. Thus, it is Fejér* monotone.  Some other properties of Fejér* monotonicity related to optimization algorithms were explored recently by \citet{Arakcheev:2025,arakcheevFejerFejerMonotonicity2025}, where Fejér* sequences were linked to Opial sequences, and in a work on Fejér monotonicity by \citet[Section 5.2]{kohlenbachFejerMonotoneSequences2023}. We refer the reader to these works for further details. 

We highlight that the concept of Fejér* monotonicity is not novel, it has appeared especially when  one is dealing with approximations of the feasible set from the inside, and in general is applied in the context of obtaning finite convergence results. Our work sheds some light on this concept, exploring its properties and differences with other weaker notions of Fejér monotonicity. This study will be carried out in the next sections.

\section{Properties of Fejér* monotonicity}
\label{sec:fejer*}

The first result shows some basic properties of Fejér* monotonicity. The result was already proved in \cite[Prop.~2.5]{Behling:2024c}. However, we chose to present the proof here, for the sake of completeness.

\begin{proposition}
    \label[proposition]{prop.Fejer*monotone.char_1} 
    Let $(x^k)_{k\in \NN}\subset\HH$ be  Fejér* monotone  with respect to a nonempty set $M$ in $\HH$. Then,

     \begin{listi}
        
\item $(x^k)_{k\in \NN}$ is bounded;
\item for every $x\in M$, the scalar sequence $(\norm{x^k-x})_{k\in\NN}$ converges;
\item  $(x^k)_{k\in \NN}$ is Fejér* monotone with respect to $\conv(M)$.

\end{listi}
\end{proposition}
  
\begin{proof}
    For proving \textbf{(i)}, take any point $x\in M$. From the definition of Fejér* monotonicity, we conclude that $x^k$ belongs to the ball with center at $x$ and radius $\lV x^{N(x)}-x\rV$ for all $k\ge N(x)$. Consequently, $(x^k)_{k\in \NN}$ is bounded.
   
    Item \textbf{(ii)} is a direct consequence of   \cref{eq:def.Fejer*monotone}, since the $N(x)$-tail of sequence $(\norm{x^k-x})_{k\in\NN}$  is monotone and bounded, so the sequence converges.

    For item \textbf{(iii)}, take any $w\in \conv(M)$. Thus, $w$ can be written as  $w = \sum_{i=1}^p \lambda_i w^i$, where $w^i\in M$,  $\lambda_i \in [0,1]$ for all $i = 1,2,\ldots,p$, $\sum_{i=1}^p \lambda_i =1$ and $p\in \NN$. Taking into account that  $(x^k)_{k\in \NN}$ is Fejér* monotone with respect to $M$, for each $i=1, 2,\ldots,p$, there exists $N(w^i)$ such that $\lV x^{k+1}-w^i\rV\le\lV x^k -w^i\rV$ for all $k\ge N(w^i)$. Squaring both sides of the last inequality and rearranging the terms, we get
    \begin{align}   \label{eq:Fejer*monotone_foo}
        0 &\leq \lV{x^k}\rV^2 - \lV{x^{k+1}}\rV^2  - 2 \scal{x^{k}-x^{k+1}}{w^i}.
    \end{align}
The above inequality also holds for all $k\ge K \coloneqq \max\{N(w ^1), N(w ^2), \ldots, N(w ^p)\}$. Multiplying \cref{eq:Fejer*monotone_foo}, for each $i = 1,2,\ldots,p$, by the respective $\lambda_i$ and adding up, we obtain
    \begin{align}   
        0 
        &\le \lV{x^k}\rV^2 -\lV{x^{k+1}}\rV^2  - 2 \scal{x^{k}-x^{k+1}}{\sum_{i=1}^p \lambda_i w^i} \\ 
        &= \lV{x^k}\rV^2 - \lV{x^{k+1}}\rV^2  - 2 \scal{x^{k}-x^{k+1}}{ w},
    \end{align}
    and therefore, $\lV x^{k+1}-w\rV\le\lV x^k - w\rV$, for all $k\ge K$. So $(x^k)_{k\in \NN}$ is Fejér* monotone with respect to $\conv(M)$, and the proof is complete.
\end{proof}

\begin{remark}
    \label{rem:Fejer*monotone.char_1}
We remark that a Fejér monotone sequence  with respect to a nonempty set  $M \subset \HH$ is also Fejér monotone with respect to $\closure{\conv(M)}$,  the closure of $\conv (M)$; see \citet[Lemma 2.1.17]{Cegielski:2012}  and \citet[Proposition 3.4]{Combettes:2001a}, or even \citet[Lemma 2.1(ii)]{Bauschke:2023}. Furthermore,  \Cref{prop.Fejer*monotone.char_1}(iii) guarantees that a Fejér* monotone sequence with respect to $M$ is as well Fejér* monotone with respect to $\conv(M)$. 
However, Fejér* monotonicity  cannot always be extended  to $\closure{\conv(M)}$, as shown in the next example.
\end{remark}

\begin{example}
\label[example]{example:Fejer*notFejer}
Let $M\coloneqq  (w^k)_{k\in\NN}\subset \RR^2$ with $w^k \coloneqq (1/2^ k,0)$, for all $k \in\NN$. Define the sequence 
$(x^k)_{k\in\NN}\subset\RR^2$ by $x^0\coloneqq (0,2)$ and, for $ \ell\in\NN$, 
\[
x^{2 \ell+1}\coloneqq x^{2 \ell}+(1/2^ \ell,0)\;,\quad 
x^{2 \ell+2}\coloneqq \left(0,\sqrt{\|x^{2 \ell+1}-(1,0)\|^2-1}\right).
\]
If $k$ is even, $x^{k+1}$  is obtained by moving  $x^k$ horizontally to the right by an amount of $1/2^{(k/2)}$. If $k$ is odd, $x^{k+1}$ is the intersection between the vertical axis and the circle centered at $w^0 = (1,0)$ passing through $x^k$.  \Cref{fig1} illustrates the set $M$, the first nine points of sequence $(x^k)_{k\in \NN}$, the convex hull $\conv(M)$, its closure $\closure{\conv(M)}$ and the limit point $x^*$ of sequence $(x^k)_{k\in \NN}$. The sequence that appears in \Cref{fig1} as well as all the calculations and the plot itself were  coded and generated using \texttt{Julia} language~\cite{Bezanson:2017}, on a MacbookAir M4, with 24GB RAM and 10-core CPU.

\begin{figure}[htbp]
    \centering
    \includegraphics[width=.85\textwidth]{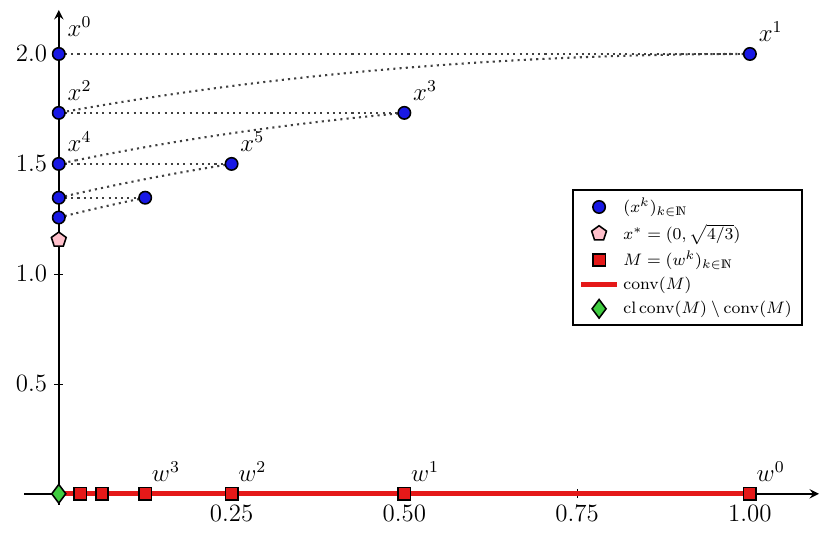}
    \caption{{From \Cref{example:Fejer*notFejer}: set $M$,  first nine points of sequence $(x^k)_{k\in \NN}$, $\conv(M)=(0,1]\times\{0\}$, $\closure{\conv(M)}\setminus\conv(M) = \{(0,0)\}$ and limit point $x^*$ of sequence $(x^k)_{k\in \NN}$.}}
    \label{fig1}
    \end{figure}

We now show that $(x^k)_{k\in\NN}$ is Fejér* monotone with respect to $M$. Denote $x^k=(\alpha_k,\beta_k)$, $k\in\NN$. One can see that 
$1\leq \beta_{k+1}\leq \beta_k$ as well as  $\|x^k-w^0\|\geq 1$, for all $k\in\NN$, so that the sequence $(x^k)_{k\in \NN}$ is well-defined.

Now, consider $\lambda \in(0,1]$ and set $w\coloneqq (\lambda,0)\in \conv(M)$. It so happens that  $\|x^{k+1}-w\|\leq\|x^k-w\|$, for all odd $k$. Indeed, 
\[
\label{seq:bk}
\beta_{k+1}=\sqrt{\alpha_k^2-2\alpha_k+\beta_k^2}\leq\sqrt{\alpha_k^2-2\alpha_k \lambda +\beta_k^2} \eqqcolon\gamma_k,
\]
and $\gamma_k$ is such that $z^k\coloneqq (0,\gamma_k)$ belongs to the circle centered in $w$ and passing through $x^k$. A simple calculation shows that  $\|z^k-w\|=\|x^k-w\|$.
Thus, 
\[
\|x^{k+1}-w\|=\sqrt{\beta_{k+1}^2+\lambda^2}\leq\sqrt{\gamma_k^2+\lambda^2}= \|x^k-w\|.
\]

{On the other hand, for $w^\ell = (1/2^\ell,0)\in {M}$ with $\ell \geq 0$, consider 
$N(w^\ell)=2\ell$. For all even $k\geq N(w^\ell)$, we have $\alpha_k=0$, 
$\alpha_{k+1}\leq 1/2^\ell$ and $\beta_{k+1}=\beta_k$. So, $\|x^{k+1}-w^\ell\|\leq\|x^k-w^\ell\|$ and the sequence is Fejér* monotone with respect to $M$.}

Note that $\conv(M)=(0,1]\times\{0\}$ and also it holds that  $\bar{w}=(0,0)\in\closure{\conv(M)}$.
We show now that $(x^k)_{k\in\NN}$ is not Fejér* monotone with respect to $\{(0,0)\}$. Indeed, for $k=2\ell$, 
$\ell\in\NN$, we have $x^{2\ell}=(0,\beta_{2\ell})$ and $x^{{2\ell}+1}=(1/2^\ell,\beta_{{2\ell}+1})=(1/2^\ell,\beta_{2\ell})$. Thus, 
\[
\|x^{{2\ell}+1}-\bar{w}\|^2=\dfrac{1}{2^{2\ell}}+\beta_{2\ell}^2>\beta_{2\ell}^2=\|x^{2\ell}-\bar{w}\|^2.
\]
Finally, just for the sake of completion, we show that sequence $(x^k)_{k\in\NN}$ converges to $x^* \coloneqq (0,\sqrt{4/3})$. In fact, for $k=2\ell$, $\ell\in\NN$, we have
 from \cref{seq:bk} that 
\begin{align}
\beta_{2\ell+2}^2 & =  \alpha_{2\ell+1}^2-2\alpha_{2\ell+1}+\beta_{2\ell+1}^2 =  \dfrac{1}{2^{2\ell}}-2\dfrac{1}{2^\ell}+\beta_{2\ell}^2 \\ 
 & =  \displaystyle\sum\limits_{j=0}^{\ell}\dfrac{1}{2^{2j}}-
2\sum\limits_{j=0}^{\ell}\dfrac{1}{2^j}+\beta_0^2  =  \displaystyle\sum_{j=0}^{\ell}\dfrac{1}{2^{2j}} \stackrel{\ell \to +\infty}{\to}\dfrac{4}{3}.
\end{align}
Clearly, we have $\norm{x^{2\ell+1}-x^{2\ell}}\to 0$ when $\ell\to+\infty$, so $x^k\to x^*$, as desired.

\end{example}

\begin{remark}
    One might argue that the set $M$  as defined in \Cref{example:Fejer*notFejer} is somehow special, for $M$ has empty interior. However, it is possible to prove the same result for an extended example where the underlying set has nonempty interior, say, for instance, the square \[S   \coloneqq \{(w',w'')\in \RR^2  \mid 0< w' \leq 1, -1 < w'' \leq 0 \} ,\] instead of $M$.
    Note that $\conv(M)  \subsetneq  (0, 1] \times (-1, 0]  = S$.
\end{remark}

We present now two counter-examples showing that convex combinations of Fejér* monotone sequences may fail to be Fejér* monotone, as well as the fact that compositions of Fejér* monotone sequences may fail to be Fejér* monotone. 

We start with the latter.
 By \emph{composition of Fejér* monotone sequences} we mean the following. Let $T_1, T_2: \HH \to \HH$ be two operators that generate Fejér* sequences with respect to a target set $C \subset \HH$,
 that is, sequences $(x^k_i)_{k\in\mathds{N}}$ defined by $x^{k+1}_i = T_i(x^k_i)$ with starting point $x^0_i \in \HH$, for $i=1,2$, that  are Fejér* monotone with respect to $C$. Consider the operator $T = T_1 \circ T_2$ and the sequence $(x^k)_{k\in\mathds{N}}$ defined by $x^{k+1} = T(x^k) = T_1(T_2(x^k))$ with starting point $x^0 \in \HH$. The question is whether the sequence $(x^k)_{k\in\mathds{N}}$ is Fejér* monotone with respect to $C$. The next example shows that this is not necessarily the case.

\begin{example}
    \label[example]{example:Fejer*notComposition}

    Let $\HH = \mathds{R}$ and let the target set be $C = \{0\}$. We define three auxiliary sequences for $k \in \mathds{N}$:
\[
u^k = 4^{-k}, \quad v^k = 2 \cdot 4^{-k}, \quad w^k = -4^{-(k+1)}.
\]
We define the operators $T_1, T_2: \mathds{R} \to \mathds{R}$ by specifying their action in the following manner:
\begin{align}
T_1(x) &= \begin{cases} 
v^k & \text{if } x = u^k, \\
0 & \text{if } x = v^k, \\
u^{k+1} & \text{if } x = w^k, \\
0 & \text{otherwise},
\end{cases} 
&
T_2(x) &= \begin{cases} 
u^k & \text{if } x = u^k, \\
w^k & \text{if } x = v^k, \\
0 & \text{if } x = w^k, \\
0 & \text{otherwise}.
\end{cases}
\end{align}

We claim now that both $T_1$ and $T_2$ generate Fej\'er* sequences with respect to the target set $C = \{0\}$.
Indeed, for $T_1$, a trajectory starting at $u^k$ moves $u^k \to v^k \to 0$. Although $|v^k| > |u^k|$ (an expansion), the sequence hits the target in 2 steps. A trajectory starting at $w^k$ moves $w^k \to u^{k+1} \to v^{k+1} \to 0$. The remaining trajectories go directly to the target. In all cases,
the sequence enters the target set after finitely many steps, so the tail is monotone (constant zero). Thus, $T_1$ generates Fej\'er* sequences.

For $T_2$, a trajectory starting at $v^k$ moves $v^k \to w^k \to 0$. Since $|w^k| < |v^k|$, this is strictly contractive. Trajectories starting at $u^k$ stay fixed at $u^k$ (monotone). The remaining trajectories go directly to the target. Thus, $T_2$ also yields Fej\'er* sequences.

Consider now the operator $T = T_1 \circ T_2$ and the sequence $(x^k)_{k\in\mathds{N}}$ defined by $x^{k+1} = T(x^k)$ with starting point $x^0 = u^0 = 1$.
Then, the sequence $(x^k)_{k\in\mathds{N}}$ evolves as follows:
\[u^0 \to v^0 \to u^1 \to v^1 \to u^2 \to v^2 \to \dots\]
Explicitly, we analyse the transitions. 

From $u^k$ to $v^k$, we have $T_2 (u^k) = u^k$, and then $T_1(u^k) = v^k$.  Thus, if $x^k = u^k$, then $x^{k+1} = v^k$. Here, the distance to the target  $C$ changes from $|u^k|$ to $|v^k| = 2|u^k|$, yielding \( |x^{k+1} - 0| > |x^k - 0|\), a violation of Fej\'er* monotonicity.

Now, from $v^k$ to $u^{k+1}$ we get $T_2 (v^k) = w^k$, and then $T_1(w^k) = u^{k+1}$. Therefore, if $x^{k+1} = v^k$, then $x^{k+2} = u^{k+1}$. Now, the distance changes from $|v^k|$ to $|u^{k+1}| = \frac{1}{8}|v^k|$, which corresponds to a contraction.

    This indicates that the sequence generated by $T$ oscillates indefinitely between expansion and contraction. Since the distance strictly increases at every step mapping $u^k \to v^k$ (which occurs infinitely often), there exists no index $K$ such that the tail of the sequence is monotone non-increasing. Therefore, the composition $T$ does not generate a Fej\'er* sequence.

The essence of this example lies in the fact that both $T_1$ and $T_2$ belong to the class of operators such that iterating them 
gives rise to a constant operator after a finite number of steps, and this class is not closed under compositions.

\end{example}

We now address the case of convex combinations of Fejér* monotone sequences. 

\begin{example}
    \label[example]{example:Fejer*notConvexCombination}
Let $\mathcal{H} = \mathds{R}$, $C = \{0\}$  and define sequences 
$(x^k)_{k\in\mathds{N}}$ and $(y^k)_{k\in\mathds{N}}$ as: $x^k\coloneqq \frac{1}{k}$ if $k$ is odd, $x^k\coloneqq\frac{1}{k+1}$ if $k$ is even, and define
$y^k\coloneqq -\frac{1}{k}$ for all $k$, so that both $(x^k)_{k\in\mathds{N}}$ and $(y^k)_{k\in\mathds{N}}$ converge monotonically to $0$; note that $(x^k)_{k\in\mathds{N}}$  os nonincreasingly and  $(y^k)_{k\in\mathds{N}}$ is increasingly,
and hence both are Fejér monotone with respect to $C$. Define 
$z^k\coloneqq \frac{1}{2}(x^k+y^k)$, a convex combination of $(x^k)_{k\in\mathds{N}}$ and $(y^k)_{k\in\mathds{N}}$. 
Observe that $z^k=0$ if $k$ is odd, $z^k=\frac{1}{2k(k+1)}$ if $k$ is even, and it follows that $(z^k)_{k\in\mathds{N}}$ is nonnegative and converges to $0$, but not monotonically, 
so that it is not Fejér monotone with respect to $C$, an hence not Fejér* monotone with respect to $C$.
\end{example}

Note that the two examples above apply to both the classical Fejér and Fejér* monotonicity because when $C$ is finite, these two notions coincide (after disregarding a finite number of indices).

The next remark gives a sufficient condition for convex combinations of Fejér* monotone sequences to be Fejér* monotone.

\begin{remark}
    \label[remark]{rem:Fejer*convexCombination}
    Fej\'er* monotonicity property is preserved under \emph{synchronized linear contraction}, that is,  if the sequences
    $(x^k)_{k\in \mathds{N}}$ and $(y^k)_{k\in \mathds{N}}$    
    approach the target $z$ along the same direction with the same step-size sequence,  then any convex combination of them is also Fej\'er* monotone. Formally, let $C$ be a closed convex set and $(x^k)_{k\in \mathds{N}}$, $(y^k)_{k\in \mathds{N}} \subset \mathcal{H}$
    be sequences. Suppose that there exists $z \in C$, $K_1,K_2 \in \mathds{N}$, and scalars $\rho_k \in [0, 1]$ such that
\[
x^{k+1} - z = \rho_k (x^k - z), \forall k \geq K_1, \quad \text{and} \quad y^{k+1} - z = \rho_k (y^k - z), \forall k \geq K_2.
\]
Then, any sequence $(z^k)_{k\in \mathds{N}}$ defined as the    
convex combination $z^k = \lambda x^k + (1-\lambda)y^k$, for $\lambda \in [0,1]$, is Fej\'er monotone for $k \ge K \coloneqq \max(K_1,K_2)$, and thus constitutes a Fej\'er* sequence. In fact, for $k \ge K$,  we have
\[
z^{k+1} - z = \lambda \rho_k (x^k - z) + (1-\lambda) \rho_k (y^k - z) = \rho_k (z^k - z).
\]
Taking norms yields $\|z^{k+1} - z\| = \rho_k \|z^k - z\| \le \|z^k - z\|$, for $k\geq K$.
\end{remark}

This section ends with  the next proposition, where the closure of the convex hull of the underlying set plays a role. Even though a Fejér* monotone sequence with respect to $M\subset \HH$ is not Fejér* monotone with respect to $\closure{\conv(M)}$, one can  characterize Fejér* monotonicity in terms of the convergence of scalar sequences and inner product sequences to the closure of the convex hull.

\begin{proposition}
    \label[proposition]{prop.Fejer*monotone.char_2} 
    Let $(x^k)_{k\in \NN}\subset\HH$ be  Fejér* monotone  with respect to a nonempty set $M$ in $\HH$. Then,
     \begin{listi}
\item for every $\Bar{x} \in \closure{\conv (M)}$ the scalar sequence $\left(\|{x^k-\Bar{x}}\|\right)_{k\in \NN}$ converges;

\item for every $\Bar{x}_1, \Bar{x}_2 \in \closure{\conv (M)}$, it holds that the sequence $\left( \scal{\Bar{x}_1-\Bar{x}_2}{x^k}\right)_{k\in \NN}$ converges. 
\end{listi}
\end{proposition}

\begin{proof}

    Regarding item \textbf{(i)}, recall that \cref{prop.Fejer*monotone.char_1}{(iii)} implies that  $(x^k)_{k\in \NN}$ is Fejér* monotone with respect to $\conv(M)$. Thus,  for any $w\in \conv(M)$,   \cref{prop.Fejer*monotone.char_1}{(ii)} yields that the scalar sequence $(\norm{x^k-w})_{k\in \NN}$ converges. 
    
     Now, take $\Bar{x} \in \closure{\conv(M)}$, and consider a sequence 
    {$(w^\ell)_{\ell\in \NN}\subset \conv(M)$} such that
    $w^\ell \rightarrow \Bar{x}$. Using the triangle inequality, we have, for all $\ell \in \NN$,
\begin{align}
  -   \left\|w^\ell-\Bar{x}\right\| \leq 
  \left\|x^k -\Bar{x}\right\| - \left\|x^k -w^\ell\right\|  \leq \left\|w^\ell-\Bar{x}\right\|.
\end{align}
Taking $\liminf$ and $\limsup$ with respect to $k$ in the above inequalities, we obtain
\begin{align}
    -\left\|w^\ell-\Bar{x}\right\| 
    & 
    \leq \liminf_{k\to \infty} \left\|x^k -\Bar{x}\right\|-\lim_{k\to\infty} \left\|x^k -w^\ell\right\| \\
    & \leq  \limsup_{k\to \infty}
     \lV x^k - \Bar{x} \rV -\lim_{k\to\infty} \left\| x^k -w^\ell\right\|  \leq\left\|w^\ell-\Bar{x}\right\|,
    \end{align}
    because, as aforementioned, item \cref{prop.Fejer*monotone.char_1}{(ii)} says that, for every  $\ell \in \NN$,  $\lim_{k\to\infty}\left\|w^\ell-x^k\right\|$ exists.
Now, as $\ell$ goes to infinity, we have $\left\|w^\ell-\Bar{x}\right\|  \to 0$, and we conclude that \[\liminf_{k\to \infty} \left\|x^k -\Bar{x}\right\| = \limsup_{k\to\infty} \left\|x^k -\Bar{x}\right\|=\lim_{k\to \infty}\left\|x^k-\Bar x\right\|,\]
so that we have the desired result.

As for item \textbf{(ii)}, let  $\Bar{x}_1, \Bar{x}_2 \in \closure{\conv (M)}$. Then,
\begin{align}
    \scal{\Bar{x}_1 - \Bar{x}_2}{x^k} & = \scal{\Bar{x}_1 - \Bar{x}_2}{x^k - \Bar{x}_1 }  + \scal{ \Bar{x}_1 -  \Bar{x}_2}{ \Bar{x}_1} \\
    & = \frac{1}{2}\left(\left\|x^k - \Bar{x}_2 \right\|^2 - \left\| \Bar{x}_1 - \Bar{x}_2 \right\|^2 - \left\|x^k- \Bar{x}_1\right\|^2\right)  +\scal{ \Bar{x}_1 -  \Bar{x}_2}{ \Bar{x}_1}.
\end{align}
Now, item \textbf{(i)} implies that the sequences $\left(\left\|x^k - \Bar{x}_1 \right\|\right)_{k\in \NN}$ and $\left(\left\|x^k - \Bar{x}_2 \right\|\right)_{k\in \NN}$ converge. Therefore, the sequence $\left( \scal{\Bar{x}_1 - \Bar{x}_2}{x^k}\right)_{k\in \NN}$ converges.

\end{proof}

\section{Fejér* and quasi-Fejér sequences}
\label{sec:fejer-quasi-fejer}

In this section we discuss the relation between Fejér* monotonicity and  quasi-Fejér monotonicity (as addressed in \cite{Combettes:2001a}).
We use the following notation:   $\ell_{+}$ is the set of real sequences with nonnegative entries, and  $\ell^1$ is the set of real sequences  with finite $\ell_1$-norm. 

\begin{definition}[{quasi-Fejér monotonicity \cite[Def.~1.1]{Combettes:2001a}}]
    \label{def:quasi_fejer}
    Let $M\subset \HH$ {be a nonempty set}.  We say that  a sequence $(x^k)_{k\in \NN} \subset \HH$, with respect to $M$, is
    \begin{listi}
        \item \emph{quasi-Fejér of Type I} if there exists a sequence $(\varepsilon_k)_{k\in \NN} \in \ell_{+} \cap \ell^1$ such that, for all $x \in M$ and $k \in \NN$, it holds that
        \[
       \left\|x^{k+1}-x\right\| \leq\left\|x^k-x\right\|+\varepsilon_k ;
        \]
        \item \emph{quasi-Fejér of Type II} if there exists a sequence $(\varepsilon_k)_{k \in \NN} \in \ell_{+} \cap \ell^1$ such that, for all $x \in M$ and $k \in \NN$, it holds that
        \[
        \left\|x^{k+1}-x\right\|^2 \leq\left\|x^k-x\right\|^2+\varepsilon_k ;
        \]
        \item \emph{quasi-Fejér of Type III} if for all $x \in M$ there exists a sequence $(\varepsilon_k)_{k \in \NN} \in \ell_{+} \cap \ell^1$ such that, for all $k \in \NN$, it holds that
        \[
        \left\|x^{k+1}-x\right\|^2 \leq\left\|x^k-x\right\|^2+\varepsilon_k.
        \]
    \end{listi}
    
\end{definition}

It is clear that Fejér sequences are quasi-Fejér of Type I sequences. Also, one can clearly see that  Type II implies Type III. Moreover, it was established in \cite[Prop.~3.2]{Combettes:2001a} that if $M$ is bounded, then Type I implies Type II.

Quasi-Fejér of Type III sequences are the most general ones. However, due to this generality, they are hardly used for the analysis of optimization algorithms. In fact, the most common type of quasi-Fejér sequences used in the literature are  quasi-Fejér of Type II sequences; see, for instance, \cite{Alber:1998,Iusem:1994}.

The next theorem shows that Fejér* monotone sequences are also  quasi-Fejér of Type III.

\begin{theorem}
    \label{prop:Fejer*implies_quasi-Fejer}
    Let $(x^k)_{k\in \NN}\subset\HH$ be a Fejér* monotone sequence with respect to a nonempty set $M$ in $\HH$. Then, $(x^k)_{k\in \NN}$ is a quasi-Fejér of Type III sequence with respect to $M$.
\end{theorem}

\begin{proof}
    Suppose that $(x^k)_{k\in \NN}$ is Fejér* monotone with respect to $M$. Then, in view of \cref{eq:def.Fejer*monotone}, for any $x \in M$, there exists $N(x)$ such that, for all $k\ge N(x)$, it holds that
    \[
    \left\|x^{k+1}-x\right\|^2 \leq \left\|x^k-x\right\|^2.
    \]
    Let us now define the sequence $(\varepsilon_k)_{k\in \NN}$. For $k\ge N(x)$, we set $\varepsilon_k\coloneqq 0$.
    On the other hand, for $k< N(x)$,  we might have 
   \[
        \left\|x^{k+1}-x\right\| > \left\|x^k-x\right\|,
        \]
    so we define 
    \[\varepsilon_k \coloneqq  \begin{cases*}
        \left\|x^{k+1}-x\right\|^2-\left\|x^k-x\right\|^2, & \text{if $\left\|x^{k+1}-x\right\| > \left\|x^k-x\right\|$}, \\
        0, & \text{otherwise}.
    \end{cases*} \]
    Thus, we conclude that  $(\varepsilon_k)_{k\in \NN} \in \ell_+ \cap \ell^1$ and, therefore, the sequence $(x^k)_{k\in \NN}$ is quasi-Fejér of Type III with respect to $M$. 
\end{proof}

\begin{remark}[Fejér* versus quasi-Fejér sequences]
    \label{rem:Fejer*notFejer}
We note that the concepts of Fejér* and quasi-Fejér of Type III sequences clearly are not equivalent, with the latter being more general than the former. Moreover, if $M$ is a singleton, then any Fejér* monotone sequence is   quasi-Fejér of Type II, using the same arguments used in the proof of \cref{prop:Fejer*implies_quasi-Fejer}. Of course, in this case  one can easily build quasi-Fejér sequences of Type II that are not Fejér*. Nevertheless,  if a sequence is  quasi-Fejér of Type II  with respect to a general nonemtpy $M$, then it is Type II with respect to $\closure{\conv(M)}$ (see \cite[Prop.~3.4]{Combettes:2001a}). Hence,  in view of \cref{example:Fejer*notFejer}, we have a Fejér* sequence that is not  quasi-Fejér of Type II. Note that we can construct a non-singleton  set $M$, for instance, $M \coloneqq [-1,0] \times  \{0\}\subset \RR^2$, such that the same sequence $(x^k)_{k\in \NN}$ of \cref{example:Fejer*notFejer} is quasi-Fejér of Type II, but not Fejér* monotone with respect to $M$.  Therefore, more than not being equivalent, we have just  seen that  the concepts of Fejér* and quasi-Fejér of Type II \emph{do not imply each other}, that is to say, these concepts are independent.
\end{remark}

The previous theorem reveals a particular quasi-Fejér of Type III sequence that is related to optimization algorithms (as seen in \Cref{sec:fejer*optimization}), through Fejér* monotonicity. Thus, we end our discussion by presenting some characterizations of Fejér* monotonicity in terms of the weak convergence of the underlying sequence. These statements are inspired by \cite{Combettes:2001a} and could be derived from the results therein, since  Fejér* monotonicity implies quasi-Fejér of Type III monotonicity. 

From now on, for a sequence $(x^k)_{k \in \NN}\subset \HH$, we denote by $\mathfrak{W}\left(x^k\right)_{k \in \NN}$ the set of its \emph{weak cluster points}  and by $\mathfrak{S}\left(x^k\right)_{k \in \NN}$ the set of its \emph{strong cluster points}. We remind that for a nonempty set \( M \), the \textit{affine hull} of \( M \) is  
\(
\aff(M) \coloneqq \left\{ \sum_{i=1}^{p} \lambda_i x_i \mid x_i \in M, \, \lambda_i \in \RR, \, \sum_{i=1}^{p} \lambda_i = 1, p\in \NN \right\},
\)
that is, the smallest affine subspace containing \( M \).  First, we give portrayals of Fejér* monotonicity in terms of the weak convergence of the sequence.

\begin{proposition}\label[proposition]{prop.Fejer*monotone.char_3}
     Let $\left(x^k\right)_{k \in \NN}$ be  Fejér* sequence with respect to nonempty $M \subset \HH$. Then,
    \begin{listi}
\item $\mathfrak{W}\left(x^k\right)_{k \in \NN} \neq \emptyset$.
\item for any  $w', w'' \in \mathfrak{W}\left(x^k\right)_{k \in \NN}$, there exists an  $\alpha \in \RR$ such that \[M \subset\left\{y \in \HH \mid \scal{y}{w'-w''}=\alpha\right\};\]
\item  if $\closure{\aff(M)} =\HH$, 
 then $\left(x^k\right)_{k \in \NN}$ converges weakly;
\item  if $x^k \rightharpoonup \bar x \in \closure{\conv(M)} $, then the scalar sequence $\left(\left\|x^k-y\right\|\right)_{k\in \NN}$ converges for every $y \in \HH$.
\end{listi}
\end{proposition}

\begin{proof}
    Item \textbf{(i)}: the statement is valid since, due to \cref{prop.Fejer*monotone.char_1}(i), the sequence  $\left(x^k\right)_{k \in \NN}$ is bounded. 
    
    Item \textbf{(ii)}:   Take any  $y \in M$. For all $k\in \NN$, we get
    \[\label{eq:Fejer*monotone_foo_2}
  \left\|x^k-y\right\|^2-\|y\|^2=\left\|x^k\right\|^2-2\scal{y}{x^k}.
    \]
    Moreover, it follows from \cref{prop.Fejer*monotone.char_1}(ii)  that $\lim_{k\to\infty} \left(\left\|x^k-y\right\|^2-\|y\|^2\right)$ exists.

   Now, given two points $w',w'' \in \mathfrak{W}\left(x^k\right)_{k\in \NN}$, such that $x^{j_k} \rightharpoonup w'$ and $x^{\ell_k} \rightharpoonup w''$,  and considering \cref{eq:Fejer*monotone_foo_2}, we have 
    \[
    \lim_{k\to\infty} \left\|x^{j_k}\right\|^2-2\scal{y}{w'}=\lim_{k\to\infty}\left\|x^{\ell_k}\right\|^2-2\scal{y}{w''},
    \]
    and therefore 
    \[
            \scal{y}{w'-w''}=\frac{1}{2}\left(\lim_{k\to\infty} \left\|x^{j_k}\right\|^2-\lim_{k\to\infty} \left\|x^{\ell_k}\right\|^2\right).
    \]
    Setting $\alpha$ as the right-hand side of the above equation yields that  \[M \subset\left\{y \in \HH \mid \scal{y}{w'-w''}=\alpha\right\},\] as required.
    
  Item \textbf{(iii)}: In view of item (ii), if $\closure{\aff(M)}=\HH$ then, for all $w', w'' \in \mathfrak{W}\left(x^k\right)_{k \in \NN}$, there exists an $\alpha \in \RR$ such that  for all $y \in \HH$ it holds that $\scal{y}{w'-w''}=\alpha$. Consequently, $\mathfrak{W}\left(x^k\right)_{k\in \NN}$ reduces to a singleton.   Due to \cref{prop.Fejer*monotone.char_2}(i), $\left(x^k\right)_{k\in \NN}$ lies in a weakly compact set and therefore converges weakly.

  Item \textbf{(iv)}: Consider any $y \in \HH$. Then, for all $k\in \NN$, we have
    \[
    \left\|x^k-y\right\|^2=\left\|x^k-\bar x\right\|^2+2\left\langle x^k-\bar x , \bar x-y\right\rangle+\|\bar x-y\|^2.
    \]
    Using the hypothesis and  \cref{prop.Fejer*monotone.char_2} on the right-hand side of the above equation, we conclude the convergence of  $\left(\left\|x^k-y\right\|\right)_{k\in \NN}$.

\qed
\end{proof}

\begin{theorem}
 Let $\left(x^k\right)_{k \in \NN}$ be a sequence in $\HH$ and let $M$ be a nonempty subset of $\HH$. Suppose that $\left(x^k\right)_{k \in \NN}$ is Fejer* monotone with respect to $M$. Then, $\left(x^k\right)_{k \in \NN}$ converges weakly to a point in $M$ if, and only if, $\mathfrak{W}\left(x^k\right)_{k \in \NN} \subset M$.
\end{theorem}

\begin{proof} 

     The result follows directly from \cref{prop:Fejer*implies_quasi-Fejer} and  \cite[Thm.~3.8]{Combettes:2001a}.
\end{proof}

We conclude this section by presenting a pair of results that characterizes the strong convergence of Fejér* monotone sequences.

\begin{proposition}
    Let $\left(x^k\right)_{k \in \NN}$ be a sequence in $\HH$ and let $M$ be a nonempty subset of $\HH$. Suppose that $\left(x^k\right)_{k \in \NN}$ is Fejer* monotone with respect to $M$. Then,
    \begin{listi}
\item for all $w',w'' \in\mathfrak{S}\left(x^k\right)_{k \in \NN}$, it holds that  \[M \subset\left\{y \in \HH \mid\scal {y-\frac{w'+w''}{2}} {w'-w''}=0\right\};\]

\item  if $\closure{\aff(M)}=\HH$, then $\mathfrak{S}\left(x^k\right)_{k \in \NN}$ contains at most one point;
\item  the sequence converges strongly if there exist $x \in M,\left(\varepsilon_k\right)_{k\in \NN} \in \ell_{+} \cap \ell^1$, and $\rho \in(0,+\infty)$ such that for all $k\in \NN$ it holds that
\[
\left\|x^{k+1}-x\right\|^2 \leq\left\|x^k-x\right\|^2-\rho\left\|x^{k+1}-x^k\right\|+\varepsilon_k .
\]
\end{listi}
\end{proposition}

\begin{proof}

     It follows from \Cref{prop:Fejer*implies_quasi-Fejer} together with   Proposition 3.9 in  \cite{Combettes:2001a}.
\end{proof}

\begin{theorem}
    Let $\left(x^k\right)_{k \in \NN}$ be a sequence in $\HH$ and let with $M$ be a nonempty subset of $\HH$. Suppose that $\left(x^k\right)_{k \in \NN}$ is Fejer* monotone with respect to $M$.  Then, the following statements are equivalent:
\begin{listi}
    \item $\left(x^k\right)_{k \in \NN}$ converges strongly to a point in $M$;

\item $\mathfrak{W}\left(x^k\right)_{k \in \NN} \subset M$ and $\mathfrak{S}\left(x^k\right)_{k \in \NN} \neq \varnothing$;
\item $\mathfrak{S}\left(x^k\right)_{k \in \NN} \cap M \neq \varnothing$.

\end{listi}

\end{theorem}

\begin{proof}
This is a direct result of combining  \Cref{prop:Fejer*implies_quasi-Fejer} and  Theorem 3.11  in \cite{Combettes:2001a}.

\end{proof}

\section{Concluding remarks}\label{sec:conlcuding}

In this work, we further expanded the understanding of Fejér* monotonicity. We provided a characterization of Fejér* monotonicity in terms of the convergence of scalar sequences and inner product sequences to the closure of the convex hull of the underlying set. The insightful  \cref{example:Fejer*notFejer} illustrates that Fejér* monotonicity cannot, in general, be extended to the closure of the convex hull of the underlying set. We also established connections between Fejér* monotonicity and quasi-Fejér monotonicity, highlighting that Fejér* monotonicity lies outside the classical Fejér and quasi-Fejér of Type II frameworks.  Fejér* monotonicity notion has already  been used to establish convergence (usually finite) of algorithms in \cite{Iusem:1986a,fukushimaFinitelyConvergentAlgorithm1982,iusemFinitelyConvergentIterative1987,DePierro:1988a,Kolobov:2022} and, more recently, was formally defined in \cite{Behling:2024c}. The results presented in this paper could be used to shed light on the analysis of the (weak) convergence of algorithms that comply with Fejér* monotonicity, mainly to aid to prove finite convergence results.

\bmhead{Acknowledgements}
   The authors would like to thank the anonymous reviewers for their careful reading of our manuscript and their insightful comments and suggestions that helped us to improve the quality of this paper. In particular, the discussions presentend in \cref{example:Fejer*notComposition,example:Fejer*notConvexCombination,rem:Fejer*convexCombination} were raised by one of the reviewers.
    \textbf{RB} was partially supported by Brazilian agency Fundação de Amparo à Pesquisa do Estado do
    Rio de Janeiro (Grant E-26/201.345/2021); \textbf{YBC} was partially supported by the USA agency National Science
    Foundation (Grant DMS-2307328); \textbf{AAR} was partially supported by Brazilian agency Conselho Nacional de
    Desenvolvimento Científico e Tecnológico -- CNPq (Grant 307987/2023-0);
     \textbf{LRS} was also partially funded by Brazilian agencies CNPq (Grants 310571/2023-5, 407147/2023-3 and 403197/2025-2) and Fundação de Amparo ao Estado de Santa Catarina -- FAPESC (Edital 21/2024, Grant 2024TR002238). 

     \bmhead{AI use declaration} During the preparation of this work, the authors used Google Gemini (Large Language Model) for language editing, generating code, and refining the clarity of the text. The authors reviewed and verified all AI-generated content and take full responsibility for the accuracy and integrity of the publication.

\bibliography{refs}

\end{document}